\documentclass[12pt]{extarticle}
\usepackage[utf8]{inputenc}
\usepackage{amsmath,amssymb,amsfonts,amsthm,mathtools,enumerate,graphicx,tikz,geometry}

\newcommand{\tsf}{\textsf}

\newcommand{\mf}{\mathfrak}
\newcommand{\mb}{\mathbb}
\newcommand{\mc}{\mathcal}
\newcommand{\tx}{\text}
\newcommand{\mrm}{\mathrm}
\newcommand{\ds}{\displaystyle}
\newcommand{\inv}{\mathrm{inv}}
\newcommand{\des}{\mathrm{des}}
\newcommand{\sqbinom}{\genfrac{[}{]}{0pt}{}}
\newcommand{\mbf}{\mathbf}
\newcommand{\ol}{\overline}
\newcommand{\modulo}[1]{\mathrm{\enspace(mod\enspace#1)}}
\newcommand{\darga}{\mathrm{darga}}
\newcommand{\wt}{\mathrm{wt}}

\theoremstyle{plain}
\newtheorem{theorem}{Theorem}[section]
\newtheorem{proposition}[theorem]{Proposition}

\newtheorem{observation}[theorem]{Observation}
\newtheorem{conjecture}[theorem]{Conjecture}

\theoremstyle{definition}

\newtheorem{example}[theorem]{Example}
\newtheorem{question}[theorem]{Question}

\title{A study of unimodality of some combinatorial sequences and polynomials}
\author{Arjun Pawar\footnote{Dhirubhai Ambani International School, Mumbai, India.  Email: \texttt{arjunpawar2003@gmail.com}}}
\date{}

\begin{document}
	
	\maketitle
	
	\begin{abstract}
		In this article, we present a short, non-exhaustive study of an important and well-known property of combinatorial sequences -- unimodality.  We shall have a look at a sample of classical results on unimodality and related properties, and then proceed to understand the unimodality of the Gaussian polynomial in more detail.  We will look at an outline of O'Hara's proof of the unimodality of the Gaussian polynomial.  In order to grasp the challenge of the problem of obtaining an \emph{injective proof} of the unimodality of the Gaussian polynomial (which is still an open question), we make several attempts and understand where these attempts fail.
	\end{abstract}
	
	\tableofcontents
	\newpage
	
	\section{Introduction}
	
	Enumerative Combinatorics is the study of features and properties, related to cardinality, of families of finite objects which have an inherent discrete nature.  This study could be by means of direct investigation into the families, in the case where they have an obvious structure, like posets, species, etc., or by more indirect means via an encoding of such objects into more combinatorial forms, like generating functions, embeddings into lattices, etc.
	
	We are interested in a very specific enumerative combinatorial property -- unimodality, that is relevant to any families of objects whose cardinalities have a \emph{sequential} nature and can be understood well via embedding into a structured universe like a ranked poset, lattice, etc.  A finite sequence of nonnegative integers is said to be unimodal if it has only one peak, that is, the sequence \emph{rises once and falls once}.  Although seemingly a very simple notion, unimodality has connections with a plethora of other combinatorial properties like log-concavity, \(\gamma\)-nonnegativity, etc. which find applications in diverse areas like simplicial geometry, geometry of polytopes, and more.
	
	In this short article, we attempt to have a quick introduction to unimodality and related notions, and look at some classical results on the same.
	
	\paragraph{Outline of the article.}  In Section~\ref{sec:prelims}, we cover some preliminary ground about poset theory and introduce ourselves to a classical result called Sperner's theorem.  In Section~\ref{sec:unimodality}, we acquaint outselves with unimodality and some related combinatorial properties for combinatorial sequences, along with the very natural connection with similar properties for polynomials.  In Section~\ref{sec:unimodality-more}, we have a look at a very simple test for unimodality and a promising technique for proofs of unimodality.  In Section~\ref{sec:Gaussian}, we conclude by discussing the unimodality of the Gaussian polynomial in detail.
	
	\section{Preliminaries from poset theory}\label{sec:prelims}
	
	In this section, we will have a look at some essential basic facts from poset theory and a few classical results.
	
	\subsection{Basic definitions and examples}
	
	Consider a relation \(\le\) on a (nonempty, usually finite) set \(P\).  We say \((P,\le)\) is a \tsf{partially ordered set} or \tsf{poset} if
	\begin{enumerate}[(a)]
		\item  \(x\le x\), for all \(x\in P\).
		\item  If \(x,y\in P\) such that \(x\le y\) and \(y\le x\), then \(x=y\).
		\item  If \(x\le y\) and \(y\le z\), for \(x,y,z\in P\), then \(x\le z\).
	\end{enumerate}
	
	A subset \(C\subseteq P\) is called a \tsf{chain} if for any \(x,y\in C\), \(x\) and \(y\) are \tsf{comparable}, that is, either \(x\le y\) or \(y\le x\).  If \(P\) itself is a chain, then we say \(P\) is a \tsf{totally ordered set} or \tsf{toset}.
	
	For any subset \(A\subseteq P\), the pair \((A,\le)\), where \(\le\) is defined on \(P\), is also a poset.  In particular, \(C\subseteq P\) is a chain if and only if \((C,\le)\) is a toset.  Henceforth, we shall mean the same by the terminologies \emph{chain} and \emph{toset}.
	
	A subset \(A\subseteq P\) is called an \tsf{antichain} if for any \(x,y\in A\), \(x\) and \(y\) are not comparable, that is, \(x\not\le y\) and \(y\not\le x\).  Note that if \(A\subseteq P\) is not a chain, then \(A\) need not be an antichain.
	
	\begin{example}
		Let \(\mathbb{N}=\{0,1,2,\ldots\}\).  We define two partial orders on \(\mb{N}\).
		
		Define \(\le_1\) on \(\mb{N}\) as
		\[
		x\le_1 y\quad\tx{if}\quad x\le y\tx{ in the usual sense},\quad\tx{for all }x,y\in\mb{N}.
		\]
		Define \(\le_2\) on \(\mb{N}\) as
		\[
		x\le_2 y\quad\tx{if}\quad x\tx{ divides }y\tx{ in the usual sense},\quad\tx{for all }x,y\in\mb{N}.
		\]
		Then \(\le_1\) is a total order on \(\mb{N}\), while \(\le_2\) is NOT a total order on \(\mb{N}\).  We note that \(x\le_2 y\) if and only if \(x\) is a factor of \(y\), that is, every prime number that divides \(x\) must divide \(y\), with atleast the same multiplicity.
		
		We note that \(\mb{N}\) is a toset under \(\le_1\), that is, \(\mb{N}\) is a chain w.r.t \(\le_1\).  So \(\mb{N}\) does not have any subset as an antichain.  But there exist antichains in \(\mb{N}\) under \(\le_2\).  For example, any subset of primes is an antichain.
	\end{example}
	
	\begin{example}
		Let \([n]=\{1,\ldots,n\}\), for any \(n\in\mb{Z}^+\), and let \(2^{[n]}\) be the set of all subsets of \([n]\).  Note that \(\subseteq\) is a partial order on \(2^{[n]}\), that is, \(2^{[n]}\) is a partial order under inclusion.  For any \(k\in\mb{Z}^+\), we define \(\ds\binom{[n]}{k}\) to be the set of all subsets of \([n]\) with size equal to \(k\).  Then trivially, we have \(\bigg|\ds\binom{[n]}{k}\bigg|=\ds\binom{n}{k}\).
		
		The set \(\{\{1\},\{1,2\},\ldots,\{1,\ldots,n\}\}\) is a chain.  More generally, for any \(k\in[n]\), the set \(\{[1],[2],\ldots,[k]\}\) is a chain.  We now note that any collection of mutually disjoint subsets of \([n]\) is an antichain.  Another example of antichain is \(\{\{1,2\},\{1,3\},\{1,4\},\ldots,\{1,n\}\}\), for \(n\ge 2\).  We also note that \(\ds\binom{[n]}{k}\) is an antichain, for all \(k\in[n-1]\).
	\end{example}
	
	\begin{example}
		Let \(\mf{S}_n\) be the set of all permutations on \([n]\).  A \tsf{permutation} on \([n]\) is a bijection \(\sigma:[n]\to[n]\).  We can \emph{write} \(\sigma\in\mf{S}_n\) as
		\[
		\sigma=\begin{pmatrix}1&2&\cdots&n\\\sigma(1)&\sigma(2)&\cdots&\sigma(n)\end{pmatrix}.
		\]
		For example, if \(n=4\) and \(\sigma(1)=2,\sigma(2)=3,\sigma(3)=1,\sigma(4)=4\), we write
		\[
		\sigma=\begin{pmatrix}1&2&3&4\\2&3&1&4\end{pmatrix}.
		\]
		We can also write \(\sigma=\sigma_1\sigma_2\cdots\sigma_n\) (one line notation), where \(\sigma_i=\sigma(i)\), for all \(i\in[n]\).  So for the above example, we have \(\sigma=2,3,1,4\).
		
		We consider the partial order \(\le\) on \(\mf{S}_n\) defined as follows.  For \(\sigma,\pi\in\mf{S}_n\), define \(\sigma<\pi\) if there exists \(i\in[n-1]\) such that \(\pi_i>\pi_{i+1}\) and
		\[
		\sigma_1\cdots\sigma_{i-1}\sigma_i\sigma_{i+1}\sigma_{i+2}\cdots\sigma_n=\pi_1\cdots\pi_{i-1}\pi_{i+1}\pi_i\pi_{i+2}\cdots\pi_n.
		\]
		Now define for any \(\sigma,\pi\in\mf{S}_n\), \(\sigma\le\pi\) if \(\sigma=\pi\) or if there exist \(\tau_1,\ldots,\tau_k\in\mf{S}_n\) such that
		\[
		\sigma<\tau_1<\cdots<\tau_k<\pi.
		\]
		
		So if \(\sigma=2314\), then \(\sigma<\pi\), for \(\pi=2341\).  Further, if \(\alpha=2134\), then \(\alpha<\sigma\).
		
		This poset is called \tsf{inversion poset} and this partial order is called the \tsf{weak Bruhat order}.
	\end{example}
	
	\begin{example}
		A \tsf{(set) partition} of \([n]\) is a family \(\{A_1,\ldots,A_m\}\subseteq 2^{[n]}\) of nonempty subsets such that \(\bigcup_{j\in[m]}A_j=[n]\) and \(A_i\cap A_j=\emptyset\), for all \(i\ne j\).  Let \(\mf{P}_n\) be the collection of all partitions of \([n]\).  We define a partial order \(\le\) on \(\mf{P}_n\).  For any two families \(\mc{A},\mc{B}\in\mf{P}_n\), we define
		\[
		\mc{A}\le\mc{B}\quad\tx{if for every }A\in\mc{A},\tx{ there exists }B\in\mc{B}\tx{ such that }A\subseteq B.
		\]
		In other words, \(\mc{A}\le\mc{B}\) if \(\mc{A}\) is a \tsf{refinement} of \(\mc{B}\).
		
		If \(n=4\), then
		\[
		\{1,2,3,4\}\le\{12,34\}\le\{1234\}.
		\]
	\end{example}
	
	An immediate interesting question in this context is as follows.
	\begin{question}\label{ques:antichain}
		Given a finite poset \(P\), what is the largest size of any antichain in \(P\) and which are the antichains in \(P\) of the largest size?
	\end{question}
	
	\subsection{Sperner's theorem and related questions}
	
	The following is a classical theorem, which answers Question~\ref{ques:antichain} for the poset \(2^{[n]}\).
	
	\begin{theorem}[Sperner's theorem~\cite{sperner1928satz}]\label{thm:sperner}
		For any antichain \(\mc{A}\subseteq2^{[n]}\),
		\[
		|\mc{A}|\le\binom{n}{\lceil n/2\rceil}.
		\]
		Further, the only antichains of maximum size in \(2^{[n]}\) are \(\binom{[n]}{\lfloor n/2\rfloor}\) and \(\binom{[n]}{\lceil n/2\rceil}\).
	\end{theorem}
	Let us have a look at the proof of Theorem~\ref{thm:sperner} by Lubell~\cite{lubell1966short}.  The proof makes use of a famous inequality in combinatorics, namely, the LYM inequality (see~\cite[Section 1.2]{Anderson:1999918}).  We will prove the first part of the theorem, that is, the upper bound on the size of the antichain; for the second part of the theorem, see~\cite[Section 1.2]{Anderson:1999918}.
	\begin{proof}[Proof of the upper bound in Sperner's theorem]
		We first note that every maximal chain in \(2^{[n]}\) is of the form
		\[
		\{\{a_1\},\{a_1,a_2\},\ldots,\{a_1,a_2,\ldots,a_n\}\},
		\]
		for some distinct \(a_1,\ldots,a_n\) (and so \(\{a_1,\ldots,a_n\}=[n]\)).  It is now immediate that this antichain is completely determined by the permutation \(\pi=a_1a_2\cdots a_n\in\mf{S}_n\).  Thus maximal chains in \(2^{[n]}\) are in one-to-one correspondence with permutations on \([n]\).  This means the number of distinct maximal chains in \(2^{[n]}\) is \(n!\).
		
		Now let \(\mc{A}\) be any antichain in \(2^{[n]}\).  Fix any \(A\in\mc{A}\).  By appealing to the identification of maximal chains with permutations, it follows that the number of chains with the minimal element \(A\) and maximal element \([n]\) is equal to \((n-|A|)!\).  Similarly the number of chains with some singleton set as a minimal element and \(A\) as the maximal element is equal to \(|A|!\).  This means \(A\) is an element in exactly \(|A|!(n-|A|)!\) maximal chains.  Further note that distinct elements in \(\mc{A}\) cannot be elements of the same chain, since \(\mc{A}\) is an antichain.  Thus we get the inequality
		\[
		\sum_{A\in\mc{A}}|A|!(n-|A|)!\le n!.
		\]
		Rewriting the above inequality, we get
		\[
		\sum_{A\in\mc{A}}\frac{1}{\binom{n}{|A|}}\le1,\quad\tx{that is,}\quad\sum_{k=0}^n\frac{p_k(\mc{A})}{\binom{n}{k}}\le1,
		\]
		where we have defined \(p_k(\mc{A})\coloneqq|\{A\in\mc{A}:|A|=k\}|,\,0\le k\le n\).
		
		Now we note that the binomial coefficients \(\binom{n}{k},\,0\le k\le n\) satisfy the inequalities
		\[
		\binom{n}{0}\le\binom{n}{1}\le\cdots\le\binom{n}{\lfloor n/2\rfloor}=\binom{n}{\lceil n/2\rceil}\ge\cdots\ge\binom{n}{n-1}\ge\binom{n}{n}.
		\]
		Thus we get
		\[
		|\mc{A}|=\sum_{k=0}^np_k(\mc{A})=\binom{n}{\lceil n/2\rceil}\sum_{k=0}^n\frac{p_k(\mc{A})}{\binom{n}{\lceil n/2\rceil}}\le\binom{n}{\lceil n/2\rceil}\sum_{k=0}^n\frac{p_k(\mc{A})}{\binom{n}{k}}\le\binom{n}{\lceil n/2\rceil}.\qedhere
		\]
	\end{proof}
	
	A stark contrast to Sperner's theorem (Theorem~\ref{thm:sperner}) is provided by the answer to Question 2.5 for the poset \(\mf{P}_n\).  Note that for any \(\mc{A}\in\mf{P}_n\), we have \(1\le |\mc{A}|\le n\).  The \tsf{Stirling numbers of the second kind} are defined as
	\[
	\mrm{S}(n,k)=|\{\mc{A}\in\mf{P}_n:|\mc{A}|=k\}|,\quad1\le k\le n.
	\]
	For \(n\in\mb{Z}^+\), let \(\normalfont\tsf{S}_n=\max_{1\le k\le n}\mrm{S}(n,k)\) and \(\normalfont\tsf{A}_n\) be the size of the largest antichain in \(\mf{P}_n\).  Appealing to Sperner's theorem (Theorem~\ref{thm:sperner}) the following is a reasonable conjecture.
	\begin{conjecture}\label{conj:partition-antichain}
		For every \(n\in\mb{Z}^+\),
		\[
		\normalfont\tsf{A}_n\le\normalfont\tsf{S}_n.
		\]
	\end{conjecture}
	Conjecture~\ref{conj:partition-antichain} is false, by the following result of Canfield~\cite{canfield1995large}.  In fact, we get \(\lim_{n\to\infty}\tsf{A}_n/\tsf{S}_n=\infty\).
	\begin{theorem}[\cite{canfield1995large}]
		\(\normalfont\tsf{A}_n>n^{1/35}\,\normalfont\tsf{S}_n\), for sufficiently large \(n\).
	\end{theorem}
	The best answer to Question~\ref{ques:antichain} for \(\mf{P}_n\) is by Canfield~\cite{canfield1998size} (also see Graham~\cite{graham1978maximum}), as follows.
	\begin{theorem}[\cite{canfield1998size}]
		There exist constants \(c_1,c_2>0\) such that for all \(n>1\),
		\[
		\frac{c_1n^a}{(\log n)^{a+1/4}}\,\normalfont\tsf{S}_n\le\normalfont\tsf{A}_n\le \frac{c_2n^a}{(\log n)^{a+1/4}}\,\normalfont\tsf{S}_n,
		\]
		where \(a=(2-e\log 2)/4\).
	\end{theorem}

	\section{Unimodality and related properties, and connections with polynomials}\label{sec:unimodality}
	
	In this section, we will collect some more properties of interest of combinatorial sequences.
	
	\subsection{Unimodality in ranked posets}
	
	Let \(P\) be a poset.  For \(x,y\in P\), we say \(y\) \tsf{covers} \(x\), denoted by \(x\lessdot y\) if \(x<y\) and there is no \(z\in P\) such that \(x<z<y\).  In other words, \(x\lessdot y\) if \(x<y\) and \(x,y\) are consecutive.  A function \(\rho:P\to\mb{N}\) is called a \tsf{rank function} on \(P\) if
	\begin{enumerate}[(a)]
		\item  \(\rho\) is monotonic with respect to \(\le\), that is, either \(\rho\) is increasing on \(P\) or \(\rho\) is decreasing on \(P\).
		\item  for any \(x,y\in P\), if \(x\lessdot y\), then \(|\rho(x)-\rho(y)|=1\).
	\end{enumerate}
	We say \(P\) is \tsf{ranked} if there exists a rank function on \(P\).  Now suppose \(\rho\) is a rank function on \(P\).  Since \(P\) is finite, let \(m=\min_{x\in P}\rho(x)\) and \(M=\max_{x\in P}\rho(x)\).  Then the sets \(A(\rho)_i=\rho^{-1}(i),\,i\in\{m,m+1,\ldots,M\}\) is a partition of \(P\).  Further, each \(A(\rho)_i\) is an antichain.
	
	The following observation is immediate.
	\begin{observation}
		Let \(P\) be a poset and let \(\rho,\lambda\) be two rank functions on \(P\).  Then there exist \(d\in\mb{Z}^+\) and \(a\in\{-1,1\}\) such that \(\lambda=a\rho+d\).
	\end{observation}
	
	The following results are classical and we mention them without reference.
	\begin{theorem}[Folklore]
		\begin{enumerate}[(a)]
			\item  The poset \(2^{[n]}\) is ranked with rank function \(\#\) defined by
			\[
			\#(A)=|A|,\quad\tx{for all }A\subseteq[n].
			\]
			\item  The poset \(\mf{S}_n\) is ranked with rank function \(\inv\) (called the \tsf{inversion statistic}) defined by
			\[
			\inv(\pi)=|\{(i,j)\in[n]^2:i<j,\,\pi_i>\pi_j\}|,\quad\tx{for all }\pi\in\mf{S}_n.
			\]
			\item  The poset \(\mf{P}_n\) is ranked with rank function \(\#\) defined by
			\[
			\#(\mc{A})=|\mc{A}|,\quad\tx{for all }\mc{A}\in\mf{P}_n.
			\]
		\end{enumerate}
	\end{theorem}
	
	Let \(P\) be a poset with a rank function \(\rho\).  Consider the antichains \(A(\rho)_i=\rho^{-1}(i),\,i\in\{0,1,\ldots,M\}\), where \(M=\max_{x\in P}\rho(x)\).  Define \(P\) to be \tsf{unimodal} if there exists \(T\in\{0,1,\ldots,M\}\) such that
	\[
	|A(\rho)_0|\le|A(\rho)_1|\le\cdots\le|A(\rho)_T|\ge|A(\rho)_{T+1}|\ge\cdots\ge|A(\rho)_M|.
	\]
	By the following inequalities satisfied by the binomial coefficients,
	\[
	\binom{n}{0}\le\binom{n}{1}\le\cdots\le\binom{n}{\lfloor n/2\rfloor}=\binom{n}{\lceil n/2\rceil}\ge\cdots\ge\binom{n}{n-1}\ge\binom{n}{n},
	\]
	it is clear that \(2^{[n]}\) is unimodal.  A simple corollary of a fundamental result by Stanley~\cite{stanley1980weyl} is that the inversion poset \(\mf{S}_n\) with the weak Bruhat order is unimodal.  At the end of this section, with some more results from the literature at our disposal, we would see that the poset \(\mf{P}_n\) is also unimodal.

	\subsection{Unimodality, log-concavity and real-rootedness}
	
	Consider a finite sequence \((a_k)_{k=0}^n\) of real numbers.  We say \((a_k)_{k=0}^n\) is \tsf{unimodal} if there exists \(T\in\{0,\ldots,n\}\) such that
	\[
	a_0\le\cdots\le a_T\ge a_{T+1}\ge a_n.
	\]
	We say \((a_k)_{k=0}^n\) is \tsf{log-concave} if \(a_k^2\ge a_{k-1}a_{k+1}\), for \(1\le k\le n-1\).  It is easy to check that a nonnegative log-concave sequence is unimodal.
	\begin{proposition}\label{prop:log-conc--unimodal}
		Let \((a_k)_{k=0}^n\) be nonnegative and log-concave.  Then \((a_k)_{k=0}^n\) is unimodal.
	\end{proposition}
	\begin{proof}
		Suppose \((a_k)_{k=0}^n\) is not unimodal.  Then there exists \(1\le r\le n-1\) such that \(a_{r-1}>a_r<a_{r+1}\).  So \(a_r^2<a_{r-1}a_{r+1}\), which is not possible since \((a_k)_{k=0}^n\) is log-concave.  Thus \((a_k)_{k=0}^n\) is unimodal.
	\end{proof}
	
	Now consider any polynomial \(f(X)=a_0+a_1X+\cdots+a_nX^n\in\mb{R}[X]\).  We say \(f(X)\) is \tsf{unimodal} (\tsf{log-concave}) if \((a_k)_{k=0}^n\) is
	unimodal (log-concave).  In this sense, Proposition~\ref{prop:log-conc--unimodal} states that if a polynomial \(f(X)\in\mb{R}[X]\) is log-concave and has nonnegative coefficients, then \(f(X)\) is unimodal.
	
	For any \(f(X)\in\mb{R}[X]\), by the Fundamental Theorem of Algebra, \(f(X)\) has all its roots in \(\mb{C}\).  We say \(f(X)\) is \tsf{real-rooted} if \(f(X)\) has all its roots in \(\mb{R}\); in other words, \(f(X)\) is real-rooted if
	\[
	a\in\mb{C},\,f(a)=0\quad\implies\quad a\in\mb{R}.
	\]
	The following is a classical result by Newton (see Comtet~\cite[Chapter 7]{comtet2012advanced}).
	\begin{theorem}[Newton]\label{thm:real-root--log-conc}
		If \(f(X)\in\mb{R}[X]\) is real-rooted, then \(f(X)\) is log-concave.
	\end{theorem}
	Applying Proposition~\ref{prop:log-conc--unimodal} and Theorem~\ref{thm:real-root--log-conc} in succession, we get the first \emph{test of unimodality} of combinatorial sequences.
	\begin{theorem}[Br\"and\'en~\cite{branden2015unimodality}]\label{thm:real-root--unimodal}
		Let \(f(X)\in\mb{R}[X]\) be real-rooted with nonnegative coefficients.  Then \(f(X)\) is unimodal.
	\end{theorem}
	\begin{proof}
		Follows immediately from Proposition~\ref{prop:log-conc--unimodal} and Theorem~\ref{thm:real-root--log-conc}.
	\end{proof}
	For any \(n\in\mb{Z}^+\), define
	\[
	B_n(X)=\sum_{k=1}^n\tsf{S}(n,k)X^k\quad\tx{and}\quad f_n(X)=e^XB_n(X).
	\]
	We can then apply Theorem~\ref{thm:real-root--unimodal} to \(f_n(X)\) and conclude that \((\tsf{S}(n,k))_{k=1}^n\) is unimodal, that is, the poset \(\mf{P}_n\) is unimodal.
	
	The following result of Stanley~\cite{stanley1989log} is very useful in proving the log concavity of more intricate combinatorial sequences.
	\begin{theorem}[Stanley~\cite{stanley1989log}]\label{thm:stanley-log-conc}
		If \(f(X),g(X)\in\mb{R}[X]\) are both log-concave with nonnegative coefficients, then \(f(X)g(X)\) is log-concave with nonnegative coefficients.
	\end{theorem}

	\subsection{Gaussian polynomials}\label{subsec:Gaussian}
	
	For any \(k\in\mb{N}\), define the polynomials
	\[
	[k]_X=\begin{cases}
	\sum_{j=0}^{k-1}X^j=\dfrac{X^k-1}{X-1},&k\ge 1\\
	0,&k=0
	\end{cases}\quad\tx{and}\quad[k]_X!=\begin{cases}
	\prod_{m=1}^{k-1}[m]_X,&k\ge1\\
	1,&k=0
	\end{cases}
	\]
	Note that for any \(m\in\mb{N}\), the polynomial \([m]_X\) is trivially log-concave and has nonnegative coefficients.  So by Theorem~\ref{thm:stanley-log-conc}, for any \(k\in\mb{N}\), the polynomial \([k]_X!\) is log-concave and has nonnegative coefficients.  Thus by Proposition~\ref{prop:log-conc--unimodal}, \([k]_X!\) is unimodal, for every \(k\in\mb{N}\).
	
	For any \(n\in\mb{N},\,0\le k\le n\), define the \tsf{Gaussian polynomial} as
	\[
	\sqbinom{n}{k}_X=\frac{[n]_X!}{[k]_X![n-k]_X!}.
	\]
	Clearly the sequence \((\sqbinom{n}{k}_1)_{k=0}^n\) is unimodal, since these are the binomial coefficients.  Stanley~\cite{stanley1989log} proves that the sequence \((\sqbinom{n}{k}_q)_{k=0}^n\) is unimodal, for all \(q\ge 1\).  It is also easy to check that \((\sqbinom{n}{k}_q)_{k=0}^n\) is also log-concave, for all \(q\ge 1\).
	
	The more difficult result to prove is that the Gaussian polynomial \(\sqbinom{n}{k}_X\) is unimodal, for \(0\le k\le n,\,n\in\mb{N}\).  This was proved via algebraic techniques by Stanley~\cite{stanley1989log}.  Further, subverting expectations, the polynomials \(\sqbinom{n}{k}_X,\,0\le k\le n,\,n\in\mb{N}\) need not be log-concave.  For example, \(\sqbinom{4}{2}_X=1+X+2X^2+X^3+X^4\) is not log-concave.  We will be revisiting the unimodality of the Gaussian polynomials in Section~\ref{sec:Gaussian}.

	\subsection{Unimodality and \(\gamma\)-nonnegativity}
	
	A finite sequence \((a_k)_{k=0}^n\) of real numbers is said to be \tsf{palindromic with center \(n/2\)} if \(a_k=a_{n-k}\), for \(0\le k\le n\).  Further, we say a polynomial \(f(X)=a_0+a_1X+\cdots+a_nX^n\in\mb{R}[X]\) is \tsf{palindromic with center \(n/2\)} if \((a_k)_{k=0}^n\) is palindromic with center \(n/2\).  The following observation is immediate.
	\begin{observation}
		A polynomial \(f(X)\in\mb{R}[X]\) is palindromic with center \(n/2\) if and only if \(X^nf(1/X)=f(X)\).
	\end{observation}
	\begin{proof}
		Let \(f(X)=a_0+a_1X+\cdots+a_nX^n\in\mb{R}[X]\).  Suppose \(f(X)\) is palindromic with center \(n/2\).  So \(a_k=a_{n-k}\), for \(0\le k\le n\).  We then have
		\[
		X^nf\bigg(\frac{1}{X}\bigg)=X^n\sum_{k=0}^n\frac{a_k}{X^k}=\sum_{k=0}^na_kX^{n-k}=\sum_{k=0}^na_{n-k}X^{n-k}=\sum_{k=0}^na_kX^k=f(X).
		\]
		Conversely suppose \(X^nf(1/X)=f(X)\).  So we have
		\[
		X^nf\bigg(\frac{1}{X}\bigg)=X^n\sum_{k=0}^n\frac{a_k}{X^k}=\sum_{k=0}^na_kX^{n-k}=\sum_{k=0}^na_{n-k}X^k\quad\tx{and}\quad f(X)=\sum_{k=0}^na_kX^k.
		\]
		Comparing coefficients, we get \(a_k=a_{n-k}\), for \(0\le k\le n\), that is, \(f(X)\) is palindromic with center \(n/2\).
	\end{proof}
	The following is an interesting characterization of palindromic polynomials.
	\begin{observation}\label{thm:gamma-basis}
		For any \(n\in\mb{N}\), a polynomial \(f(X)\in\mb{R}[X]\) is palindromic with center \(n/2\) if and only if there exist unique \(\gamma_k\in\mb{R},\,0\le k\le\lfloor n/2\rfloor\) such that
		\[
		f(X)=\sum_{k=0}^{\lfloor n/2\rfloor}\gamma_kX^k(1+X)^{n-2k}.
		\]
	\end{observation}
	If \(f(X)\in\mb{R}[X]\) is palindromic with center \(n/2\) and the coefficients \(\gamma_k,\,0\le k\le\lfloor n/2\rfloor\) are as in Observation~\ref{thm:gamma-basis}, define \(\gamma^n(f)=(\gamma_k)_{k=0}^{\lfloor n/2\rfloor}\).  If \(\gamma^n(f)\) is nonnegative, we say that \(f(X)\) is \tsf{\(\gamma\)-nonnegative}.  An important result of Br\"and\'en~\cite{branden2008actions} is as follows.
	\begin{theorem}[Br\"and\'en~\cite{branden2008actions}]\label{thm:real-root--gamma-nonneg}
		If \(f(X)\in\mb{R}[X]\) is real-rooted with nonnegative coefficients, then \(f(X)\) is \(\gamma\)-nonnegative.
	\end{theorem}
	The following result relates \(\gamma\)-nonnegativity with unimodality.
	\begin{theorem}[Petersen~\cite{petersen2015eulerian}]\label{thm:petersen}
		Every \(\gamma\)-nonnegative polynomial is unimodal.
	\end{theorem}
	
	For any \(n\in\mb{Z}^+,\,\pi\in\mf{S}_n\), define the \tsf{descent statistic} as \(\des(\pi)=|\{i\in[n]:\pi_i>\pi_{i+1}\}|\).  The polynomial
	\[
	A_n(X)=\sum_{\pi\in\mf{S}_n}X^{\des(\pi)},
	\]
	is called the \tsf{Eulerian polynomial}.  We can write \(A_n(X)=\sum_{k=0}^{n-1}A_{n,k}X^k\), where \(A_{n,k}\coloneqq|\{\pi\in\mf{S}_n:\des(\pi)=k\}|,\,1\le k\le n-1\).  So \(A_n(X)\) has nonnegative coefficients.  It can be checked that \(A_n(X)\) is palindromic.  Further, it is a classical result of Fr\"obenius that \(A_n(X)\) is real-rooted.  Thus \(A_n(X)\) is \(\gamma\)-nonnegative.  So by Theorem~\ref{thm:petersen}, \(A_n(X)\) is unimodal.

	\section{More about unimodality}\label{sec:unimodality-more}
	
	In this section, we will cover some more results about unimodality of polynomials.
	
	\subsection{A test of unimodality}
	
	We look at a simple test of unimodality by Boros and Moll~\cite{boros1999unimodality}, and some consequences.
	
	Consider a polynomial \(f(X)=a_0+a_1X+\cdots+a_nX^n\in\mb{R}[X]\).  We say \(f(X)\) has \tsf{has nondecreasing coefficients} if \(a_0\le a_1\le\cdots\le a_n\).  If \(f(X)\) is unimodal, we define the \tsf{mode} of \(f(X)\) to be the least \(k_0\in\{0,\ldots,n\}\) satisfying \(a_0\le a_1\le\cdots\le a_{k_0}\ge a_{k_0+1}\ge\cdots\ge a_n\).  The following observations are immediate and easy to check.
	\begin{observation}\label{obs:unimodal-prop}
		Let \(f(X),g(X)\in\mb{R}[X]\) have nonnegative coefficients.
		\begin{enumerate}[(a)]
			\item  \(f(X)\) is unimodal with mode \(k_0\) if and only if \(Xf(X)\) is unimodal with mode \(k_0+1\).
			\item  If \(f(X),g(X)\) are unimodal with the same mode \(k_0\), then for any \(\alpha,\beta\ge0\), \(\alpha f(X)+\beta g(X)\) is unimodal.
		\end{enumerate}
	\end{observation}

	\begin{theorem}[Boros, Moll~\cite{boros1999unimodality}]\label{thm:unimodality-test}
		If \(f(X)\in\mb{R}[X]\) has nonnegative nondecreasing coefficients, then \(f(X+1)\) is unimodal.
	\end{theorem}
	\begin{proof}
		For \(0\le r\le m\), define
		\[
		P_{m,r}(X)=(1+X)^{m+1}-(1+X)^r.
		\]
		Then it is easy to check the following recurrence.
		\[
		P_{m+1,r}(X)=P_{m,r}(X)+X(1+X)^{m+1}.
		\]
		Using this recurrence and induction on \(m\ge r\), it follows that \(P_{m,r}(X)\) is unimodal with mode \(1+\lfloor m/2\rfloor\).  Now let \(f(X)=a_0+a_1X+\cdots+a_nX^n\in\mb{R}[X]\) have nonnegative nondecreasing coefficients.  We can then write
		\[
		Xf(X+1)=a_0P_{n,0}(X)+(a_1-a_0)P_{n,1}(X)+\cdots+(a_n-a_{n-1})P_{n,n}(X).
		\]
		So by Observation~\ref{obs:unimodal-prop} (b), \(Xf(X+1)\) is unimodal and then by Observation~\ref{obs:unimodal-prop} (a), \(f(X+1)\) is unimodal.
	\end{proof}
	\begin{example}
		Taking \(f(X)=X^n\) in Theorem~\ref{thm:unimodality-test} gives the unimodality of the binomial coefficients.
	\end{example}
	\begin{example}
		Fix any \(n,m\in\mb{Z}^+\).  Define, for \(0\le k\le m\),
		\[
		a_k=\sum_{j=k}^mn^j\binom{j}{k},\quad b_k=\sum_{j=k}^mj^n\binom{j}{k},\quad c_k=\sum_{j=k}^mj^j\binom{j}{k}.
		\]
		Further define the polynomials
		\[
		f(X)=\sum_{k=0}^ma_k(X-1)^k,\quad g(X)=\sum_{k=0}^mb_k(X-1)^k,\quad h(X)=\sum_{k=0}^mc_k(X-1)^k.
		\]
		Then it is easy to check that \(f(X),g(X),h(X)\) all have nonnegative nondecreasing coefficients, and so by Theorem~\ref{thm:unimodality-test}, \(f(X+1),g(X+1),h(X+1)\) are unimodal, that is, the sequences \((a_k)_{k=0}^m,(b_k)_{k=0}^m,(c_k)_{k=0}^m\) are unimodal.
	\end{example}
	\begin{example}
		Fix \(2r\) positive integers \(a_1,\ldots,a_r,n_1,\ldots,n_r\) such that \(2<a_1<\cdots<a_r\).  Define for any \(m\in\mb{Z}^+\),
		\[
		c_k=\sum_{j=k}^m\binom{a_1m}{j}^{n_1}\cdots\binom{a_rm}{j}^{n_r}\binom{j}{k}.
		\]
		Then one can check that the polynomial \(f(X)=\sum_{k=0}^mc_k(X-1)^k\) has nonnegative nondecreasing coefficients, thus implying by Theorem~\ref{thm:unimodality-test} that \((c_k)_{k=0}^m\) is unimodal.
	\end{example}
	
	\subsection{Unimodality, lattice paths and the reflection principle}
	
	Let us have a look at a promising technique to prove unimodality of combinatorial sequences, namely, reflection of lattice paths.  We follow the arguments and examples of Sagan~\cite{sagan1997unimodality}.
	
	A \tsf{lattice path} in \(\mb{Z}^2\) is any finite sequence \((v_k)_{k=0}^n\) in \(\mb{Z}^2\) such that \(\|v_{k+1}-v_k\|=1,\,0\le k\le n-1\), where \(\|\cdot\|\) denotes the Euclidean distance in \(\mb{R}^2\).  We clearly know the geometric meaning of a line in the plane.  Formally, a \tsf{line} in \(\mb{R}^2\) is any set of the form
	\[
	\ell_{a,b}\coloneqq\{ta+b:t\in\mb{R}\},\quad\tx{for some fixed }a,b\in\mb{R}^2,\,a\ne0.
	\]
	Given a line \(\ell_{a,b}\) in \(\mb{R}^2\) and any \(v\in\mb{R}^2\), the \tsf{reflection} of \(v\) about \(\ell_{a,b}\) is defined as \(r(\ell_{a,b},v)=2a-v\) (it does not depend on \(b\)).  We say the line \(\ell_{a,b}\) is \tsf{\(\mb{Z}^2\)-invariant} if \(r(\ell_{a,b},v)\in\mb{Z}^2\), for all \(v\in\mb{Z}^2\).
	
	Consider a \(\mb{Z}^2\)-invariant line \(\ell_{a,b}\) and a finite sequence \(\mbf{v}=(v_k)_{k=0}^n\) in \(\mb{Z}^2\) such that \(v_k\in\ell_{a,b}\), for some \(k\in\{0,\ldots,n\}\).  Let \(k_0=\max\{k\in\{0,\ldots,n\}:v_k\in\ell_{a,b}\}\).  We define the \tsf{reflection of \(\mbf{v}\) about \(\ell_{a,b}\)} as
	\[
	r(\ell_{a,b},\mbf{v})=(v_0,\ldots,v_{k_0},r(\ell_{a,b},v_{{k_0}+1}),\ldots,r(\ell_{a,b},v_n))=(v_0,\ldots,v_{k_0},2a-v_{k_0+1},\ldots,2a-v_n).
	\]
	
	\begin{example}
		For any \(n\in\mb{Z}^+,\,0\le k\le n\), let \(T_{n,k}\) be the set of all lattice paths \((v_j)_{j=0}^n\) satisfying \(v_0=(0,0),\,v_n=(k,n-k)\) and \(v_{j+1}-v_j\in\{(1,0),(0,1)\}\), for \(0\le j\le n-1\).  It is easy to see that \(|T_{n,k}|=\binom{n}{k}\).  We can now show the unimodality of the binomial coefficients using the reflection principle.  Fix any \(n\in\mb{Z}^+,\,k<\lfloor n/2\rfloor\).  Let \(v=(k,n-k),\,w=(k+1,n-k-1)\) and let \(\ell\) be the perpendicular bisector of the line segment \(\ol{vw}\).  Then we can check that \(\ell\) is \(\mb{Z}^2\) invariant and further, the map \(\phi:T_{n,k}\to T_{n,k+1}\) defined as
		\[
		\phi(\mbf{v})=r(\ell,\mbf{v}),\quad\tx{for all }\mbf{v}\in T_{n,k},
		\]
		is an injection.  This proves that \(\binom{n}{k}\le\binom{n}{k+1}\).
	\end{example}
	
	For any \((a,b)\in(\mb{Z}^+)^2,\,n\in\mb{Z}^+,\,n\equiv a+b\modulo{2}\), define \(F_{a,b}(n)\) to be the set of all lattice paths \((v_k)_{k=0}^n\) satisfying \(v_0=(0,0),\,v_n=(a,b)\) and \(v_{k+1}-v_k\in\{(\pm1,0),(0,\pm1)\}\), for \(0\le k\le n-1\).  DeTemple and Robertson~\cite{detemple1984equally} gave the following expression for \(|F_{a,b}(n)|\).
	\begin{theorem}[DeTemple, Robertson~\cite{detemple1984equally}]\label{thm:Fab}
		For any \((a,b)\in(\mb{Z}^+)^2,\,n\in\mb{Z}^+,\,n\equiv a+b\modulo{2}\),
		\[
		|F_{a,b}(n)|=\binom{n}{(n+a-b)/2}\binom{n}{(n-a-b)/2}.
		\]
	\end{theorem}
	The following result is now an easy consequence of Theorem~\ref{thm:Fab} and the reflection principle.
	\begin{theorem}[Sagan~\cite{sagan1997unimodality}]\label{thm:sagan-unimodal}
		For any \(n\in\mb{Z}^+,\,0\le k\le n\), the sequence
		\[
		\bigg(\binom{n}{j}\binom{n}{k-j}\bigg)_{j=0}^k
		\]
		is unimodal.
	\end{theorem}
	Very interestingly, taking \(k=2j\), for some \(0\le j\le\lfloor n/2\rfloor\), the unimodality in Theorem~\ref{thm:sagan-unimodal} says
	\[
	\binom{n}{j-1}\binom{n}{j+1}\le\binom{n}{j}^2,
	\]
	which is exactly the assertion that the binomial coefficients are log-concave.

	\section{Unimodality of Gaussian polynomials}\label{sec:Gaussian}
	
	In this concluding section, we discuss the unimodality of Gaussian polynomials in more detail.  Recall from Subsection~\ref{subsec:Gaussian} that for any \(n\in\mb{N},\,0\le k\le n\), the Gaussian polynomial is defined as
	\[
	\sqbinom{n}{k}_X=\frac{[n]_X!}{[k]_X![n-k]_X!}.
	\]
	We introduce a change of notation.  For any \(a,b\in\mb{N}\), define
	\[
	\tsf{G}_{a,b}(X)=\sqbinom{a+b}{a}_X.
	\]
	We then have
	\[
	\tsf{G}_{a,b}(X)=\frac{\prod_{i=1}^{a+b}(X^i-1)}{\prod_{j=1}^a(X^j-1)\prod_{j=1}^b(X^j-1)}=\frac{\prod_{i=1}^a(X^i-1)\prod_{i=1}^b(X^{a+i}-1)}{\prod_{j=1}^a(X^j-1)\prod_{j=1}^b(X^j-1)}=\frac{\prod_{i=1}^b(X^{a+i}-1)}{\prod_{j=1}^b(X^j-1)}.
	\]
	Let \(f(X)=\prod_{i=1}^b(X^{a+i}-1)\) and \(g(X)=\prod_{j=1}^b(X^j-1)\).  Since the coefficients are all in \(\mb{Q}\), by the division algorithm, there exist \(h(X),r(X)\in\mb{Q}[X]\) such that \(f(X)=g(X)h(X)+r(X)\), where \(r(X)=0\) or \(\deg r<\deg g\).  If \(r(X)\ne0\) and \(\deg r<\deg g\), then there exists \(M>0\) such that for all \(x>M\), we have \(|r(x)/g(x)|<1\).  So we will have \(\tsf{G}_{a,b}(X)=h(X)+r(X)/g(X)\).  But \(\tsf{G}_{a,b}(x)\) is an integer whenever \(x\) is a power of a prime (see, for eg., Lidl and Niderreiter~\cite{lidl1997finite}).  This will give us a contradiction and then we can conclude that \(r(X)=0\), thus proving that \(\tsf{G}_{a,b}(X)\) is indeed a polynomial.  Now, it is well-known that the Gaussian polynomial is unimodal.
	\begin{theorem}[Sylvester~\cite{sylvester1878xxv}, Proctor~\cite{proctor1982solution}, White~\cite{white1980monotonicity}]\label{thm:gaussian-unimodal}
		For any \(a,b\in\mb{N}\), \(\tsf{G}_{a,b}(X)\) is unimodal.
	\end{theorem}
	However, the proofs of Sylvester~\cite{sylvester1878xxv}, Proctor~\cite{proctor1982solution} and White~\cite{white1980monotonicity} are not \emph{purely combinatorial}.  Finding such a combinatorial proof of Theorem~\ref{thm:gaussian-unimodal} was open for a long time, with O'Hara~\cite{o1990unimodality} finally giving a beautiful combinatorial argument.  A key ingredient of this argument was a structure theorem for a nice poset.  Zeilberger~\cite{zeilberger1989ohara} gave an elementary exposition of this proof and further went on to give an \emph{algebraic version} of the structure theorem in \cite{zeilberger1989one}, called the KOH identity.
	
	We will now look at an outline of this combinatorial argument of the unimodality of the Gaussian polynomial.  If \(a=0\) or \(b=0\), then the unimodality of \(\tsf{G}_{a,b}(X)\) is trivial.  So let us assume that \(a,b\ge 1\).  Define
	\[
	U(a,b)=\{\lambda=(\lambda_1,\ldots,\lambda_a)\in\mb{N}^a:b\ge\lambda_1\ge\cdots\ge\lambda_a\ge0\},
	\]
	and for \(k\in\{0,\ldots,ab\}\), define
	\[
	U_k(a,b)=\bigg\{\lambda\in U(a,b):\sum_{i=1}^a\lambda_i=k\bigg\}.
	\]
	The following are useful observations.
	\begin{observation}[see Zeilberger~\cite{zeilberger1989ohara}]
		For any \(a,b\ge 1\),
		\begin{enumerate}[(a)]
			\item  \(\deg \tsf{G}_{a,b}=ab\).
			\item  \(\tsf{G}_{a,b}(X)=\sum_{k=0}^{ab}c_k(a,b)X^k\), where \(c_k(a,b)\coloneqq|U_k(a,b)|,\,0\le k\le ab\).
			\item  \(\tsf{G}_{a,b}(X)\) is palindromic, that is, \(c_k(a,b)=c_{ab-k}(a,b)\), for \(0\le k\le ab\).
		\end{enumerate}
	\end{observation}

	\subsection{O'Hara's proof of unimodality of Gaussian polynomials}
	
	The KOH identity~\cite{zeilberger1989one}, which is an algebraic version of O'Hara's structure theorem~\cite{o1990unimodality}, is as follows.
	\begin{theorem}[KOH Identity, Zeilberger~\cite{zeilberger1989one}]
		For any \(a,b\ge 1\),
		\[
		\tsf{G}_{a,b}(X)=\sum_{(d_1,\ldots,d_b):\sum_{i=1}^bid_i=b}\bigg(X^{b\big(\sum_{i=1}^bd_i\big)-b-\sum_{1\le i<j\le b}(j-i)d_id_j}\prod_{i=0}^{b-1}\tsf{G}_{a_i,b_i}(X)\bigg),
		\]
		where for \(0\le i\le b-1\), we have \(a_i=(b-i)b-2i+\sum_{j=0}^{i-1}2(i-j)d_{b-j}\) and \(b_i=d_{b-i}\).
	\end{theorem}
	
	For any polynomial \(f(X)=\sum_{k=i}^ja_kX^k\) with \(a_i\ne0,\,a_j\ne0\), define the \tsf{darga} of \(f(X)\) as \(\darga(f)=i+j\).  Further, we say \(f(X)\) is \tsf{darga-palindromic} if \(a_k=a_{j-k}\), for \(i\le k\le j\).  The following observations follow immediately from the definition.
	\begin{observation}[Zeilberger~\cite{zeilberger1989one}]\label{obs:Gaussian-prop}
		\begin{enumerate}[(a)]
			\item  If \(f(X),g(X)\) are darga-palindromic and unimodal with \(\darga(f)=\darga(g)=m\), then \(f(X)+g(X)\) is darga-palindromic and unimodal with \(\darga(f+g)=m\).
			\item  If \(f(X),g(X)\) are darga-palindromic and unimodal with \(\darga(f)=m,\,\darga(g)=n\), then \(f(X)g(X)\) is darga-palindromic and unimodal with \(\darga(fg)=m+n\).
			\item  If \(f(X)\) is darga-palindromic and unimodal with \(\darga(f)=m\), then for any \(n\in\mb{N}\), \(X^nf(X)\) is darga-palindromic and unimodal with \(\darga(X^nf(X))=n+m\).
		\end{enumerate}
	\end{observation}
	We now note that \(\tsf{G}_{a,b}(X)\) is, in fact, darga-palindromic with \(\darga(\tsf{G}_{a,b})=ab\).  We are now ready to prove Theorem~\ref{thm:gaussian-unimodal}.
	
	\begin{proof}[Proof of Theorem~\ref{thm:gaussian-unimodal}]
		As noted before, if \(a=0\) or \(b=0\), the result is trivially true.  So now by induction, assume that \(\tsf{G}_{a',b'}(X)\) is unimodal for all pairs \((a',b')\) such that \(a'<a\) or \(b'<b\); we also know that \(\darga(\tsf{G}_{a',b'})=a'b'\).  In the expression of KOH identity for \(\tsf{G}_{a,b}(X)\), for \((d_1,\ldots,d_b)=(b,0,\ldots,0)\), the RHS is equal to \(X^{b(b-1)}\tsf{G}_{a-2(b-1),b}\), which is unimodal with darga \(ab\), by the induction hypothesis and Observations~\ref{obs:Gaussian-prop} (b),(c).  For every other \((d_1,\ldots,d_b)\ne(b,0,\ldots,0)\), again the RHS can be checked to be unimodal with darga \(ab\), using the induction hypothesis and Observations~\ref{obs:Gaussian-prop} (b),(c).  Then we conclude the result using Observation~\ref{obs:Gaussian-prop} (a).
	\end{proof}
	
	\subsection{Conclusion: some attempts to get an \emph{injective proof}}
	
	It is evident that the most desirable, purely combinatorial proof of unimodality of \(\tsf{G}_{a,b}(X)\) would be an injective argument, that gives an explicit injection \(U_k(a,b)\hookrightarrow U_{k+1}(a,b)\), for all \(k<\lfloor ab/2\rfloor\).  It is still an open question to find such injections (see Zeilberger~\cite{zeilberger1989ohara}).  We conclude by listing a few attempts at obtaining such injections and observe the bottlenecks that arise.  Fix \(a,b\ge1\).
	\begin{enumerate}[(1)]
		\item  Fix any \(k<\lfloor n/2\rfloor\).  For any \(\lambda\in U_k(a,b)\), let \(j\in\{1,\ldots,a\}\) be the largest such that \(\lambda_j=b\), if there is one, otherwise set \(j=0\).  Map \(\lambda\) to \(\lambda'\), where \(\lambda'=(\lambda_1,\ldots,\lambda_j,\lambda_{j+1}+1,\lambda_{j+2},\ldots,\lambda_a)\in U_{k+1}(a,b)\).
		
		This is essentially a \emph{column-wise filling} and we see that injectivity is lost for \(k=2b-2\), for \(\lambda=(b,b-2,0,\ldots,0),\,\delta=(b-1,b-1,0,\ldots,0)\).
		\item  Fix any \(k<\lfloor n/2\rfloor\).  We now attempt a \emph{row-wise filling} which is easily described using a \emph{transpose operation}.  First we note that we can identify every element of \(U(a,b)\) with a \(b\times a\) \(\{0,1\}\)-matrix, as follows.  For any \(\lambda\in U(a,b)\), define the matrix \(M_\lambda^{(a,b)}\) as
		\[
		(M_\lambda^{(a,b)})_{i,j}=\begin{cases}
		1,&j\le\lambda_i\\
		0,&j>\lambda_i
		\end{cases}
		\]
		We can then check that for any \(\lambda\in U(a,b)\), there exists unique \(\theta\in U(b,a)\) such that \((M_\lambda^{(a,b)})^T=M_\theta^{(b,a)}\), and conversely (that is, this is a one-to-one correspondence).  Now consider any \(\lambda\in U_k(a,b)\).  Let \(\theta\in U_k(b,a)\) be obtained from \(\lambda\) via the one-to-one correspondence described above.  Let \(\theta'\in U_{k+1}(b,a)\) be defined as in (1) and let \(\lambda'\in U_{k+1}(a,b)\) be obtained from \(\theta'\) via the one-to-one correspondence described above.  Map \(\lambda\) to \(\lambda'\).
		
		Again here, injectivity is lost for \(k=2a-2\), for \(\lambda,\delta\) defined by \(\lambda'=(a,a-2,0,\ldots,0),\,\delta'=(a-1,a-1,0,\ldots,0)\).
		\item  For any \(\lambda\in U(a,b)\), define
		\[
		n_\lambda=\sum_{i=1}^a\lambda_i(b+1)^{a-i}.
		\]
		Consider any \(\lambda\in U_k(a,b)\).  Let \(\lambda'\in U_{k+1}(a,b)\) be defined by the conditions
		\[
		\lambda_i\le\lambda'_i,\,\forall\,i\in[a]\quad\tx{and}\quad n_{\lambda'}=\min\{n_\tau:\tau\in U_{k+1}(a,b)\tx{ and }\lambda_i\le\tau_i,\,\forall\,i\in[a]\}
		\]
		Here injectivity is lost for \(k=b\), for  \(\lambda=(b,0,0,\ldots,0),\,\delta=(b-1,1,0,\ldots,0)\).
		\item  For any \(\lambda\in U(a,b)\), define
		\[
		\wt(\lambda)=\max\{i\lambda_i:i\in[a]\}.
		\]
		Consider any \(\lambda\in U_k(a,b)\).  Let \(\lambda'\in U_{k+1}(a,b)\) be defined by the conditions
		\[
		\lambda_i\le\lambda'_i,\,\forall\,i\in[a]\quad\tx{and}\quad \wt(\lambda')=\max\{\wt(\tau):\tau\in U_{k+1}(a,b)\tx{ and }\lambda_i\le\tau_i,\,\forall\,i\in[a]\}
		\]
		In this case, the rule is not well-defined for \(k=1\), for
		\(\lambda=(1,0,\ldots,0\).
	\end{enumerate}

	\paragraph{Acknowledgements.}  The author would like to thank S. Venkitesh, Department of Mathematics, IIT Bombay, for several useful discussions on the topic and valuable comments in improving the presentation of this article.

	\bibliographystyle{alpha}
	\bibliography{references}
\end{document}